\documentclass[11pt,a4paper]{article}

\usepackage[T1]{fontenc}
\usepackage[utf8]{inputenc}
\usepackage{lmodern}
\usepackage[a4paper,margin=2.7cm]{geometry}
\usepackage{amsmath,amssymb,amsthm,mathtools}
\usepackage{booktabs,array,microtype}
\usepackage{caption,needspace,placeins,fancyhdr}
\AddToHook{cmd/section/before}{\FloatBarrier\Needspace{7\baselineskip}}
\AddToHook{env/theorem/before}{\Needspace{9\baselineskip}}
\AddToHook{env/lemma/before}{\Needspace{7\baselineskip}}
\AddToHook{env/proposition/before}{\Needspace{9\baselineskip}}
\usepackage{xcolor}
\usepackage[colorlinks=true,linkcolor=blue!55!black,citecolor=blue!55!black,
urlcolor=blue!55!black]{hyperref}
\usepackage[nameinlink,capitalize,noabbrev]{cleveref}

\allowdisplaybreaks[2]
\numberwithin{equation}{section}

\newtheorem{theorem}{Theorem}[section]
\newtheorem{lemma}[theorem]{Lemma}
\newtheorem{proposition}[theorem]{Proposition}
\newtheorem{corollary}[theorem]{Corollary}
\theoremstyle{remark}
\newtheorem{remark}[theorem]{Remark}

\crefname{theorem}{Theorem}{Theorems}
\Crefname{theorem}{Theorem}{Theorems}
\crefname{lemma}{Lemma}{Lemmas}
\Crefname{lemma}{Lemma}{Lemmas}
\crefname{proposition}{Proposition}{Propositions}
\Crefname{proposition}{Proposition}{Propositions}
\crefname{corollary}{Corollary}{Corollaries}
\Crefname{corollary}{Corollary}{Corollaries}
\crefname{remark}{Remark}{Remarks}
\Crefname{remark}{Remark}{Remarks}

\newcommand{\R}{\mathbb R}
\newcommand{\N}{\mathbb N}
\newcommand{\abs}[1]{\lvert#1\rvert}
\newcommand{\norm}[1]{\lVert#1\rVert}
\newcommand{\ip}[2]{\left\langle #1,#2\right\rangle}
\newcommand{\one}{\mathbf 1}
\newcommand{\ind}{\mathbf 1}
\newcommand{\dd}{\,\mathrm d}
\newcommand{\wh}{\widehat}
\DeclareMathOperator{\diver}{div}
\newcommand{\revise}[1]{{\color{black}#1}}  

\title{Toward P\'{o}lya's Conjecture: Improving the Individual Li--Yau Bound via Energy Orthogonality}
\author{Yifan Wang\footnote{School of Statistics and Mathematics, 
Central University of Finance and Economics, Beijing 102206, China (wangyifan@lsec.cc.ac.cn)}\ \ \ and \ \ 
Hehu Xie\footnote{SKLMS, NCMIS, Institute of Computational Mathematics,
Academy of Mathematics and Systems Science, Chinese Academy of Sciences,
No.~55, Zhongguancun Donglu, Beijing 100190, China, and School of
Mathematical Sciences, University of Chinese Academy of Sciences,
Beijing 100049, China (hhxie@lsec.cc.ac.cn).}}

\date{}

\begin{document}
\maketitle

\begin{abstract}
We establish two complementary lower-bound mechanisms for individual
eigenvalues of the Dirichlet Laplacian. First, energy orthogonality yields
a frequency-dependent cap on the Fourier density of a finite spectral
projection. Combining this cap with the standard $L^2$ Bessel estimate and
a radial-capacity bathtub principle gives, on every open set of finite
positive measure in $\R^n$ with $n\geq2$,
\[
 \lambda_k\geq c_n(2\pi)^2\omega_n^{-2/n}|\Omega|^{-2/n}k^{2/n},
 \qquad \frac{n}{n+2}<c_n<1.
\]
The constant $c_n$ is characterized by a scalar equation, with
$c_2=0.5383068077\ldots$. This estimate preserves Weyl scaling and
strictly improves the individual consequence of the Li--Yau sum
inequality, although it does not improve the sharp leading coefficient
in that sum inequality. Second, we retain part of the spectral deficit
discarded when an eigenvalue sum is bounded by its largest term. A lower
bound for the counting function, integrated through the exact first
Riesz-mean identity, leads to a strictly monotone scalar equation. Its
unique positive root is no weaker than the volume-only bound, and we give
necessary and sufficient criteria for strict improvement over both that
baseline and any independent lower bound. Quantitative estimates of
Jiang--Lin provide an explicit implementation on bounded Lipschitz domains.
The final comparisons and numerical example distinguish improvements
within this framework from stronger estimates available under additional
geometric or spectral assumptions.


\end{abstract}

\section{Introduction}\label{sec:introduction}

Let $\Omega\subset\R^n$ be an open set with finite positive measure
$V=|\Omega|$, where $n\geq2$. We write
$\omega_n=|B_1(0)|$ and
\begin{equation}\label{eq:weyl-constant}
 C_n=(2\pi)^2\omega_n^{-2/n}.
\end{equation}
The eigenvalues of the Dirichlet Laplacian, listed in nondecreasing
order and repeated according to multiplicity, satisfy
\[
0<\lambda_1\leq\lambda_2\leq\cdots\nearrow\infty.
\]
Weyl's law gives the asymptotic formula
\[
\lambda_k\sim C_nV^{-2/n}k^{2/n}
\qquad\text{as }k\to\infty.
\]
P\'olya's conjecture~\cite{Polya1961} asks whether the Weyl term is a
termwise lower bound,
\begin{equation}\label{eq:polya}
 \lambda_k\geq C_nV^{-2/n}k^{2/n},
\end{equation}
for every $k$. The conjecture remains open for general domains. The
classical Li--Yau inequality~\cite{LiYau1983} gives
\begin{equation}\label{eq:li-yau-sum}
 \sum_{j=1}^k\lambda_j
 \geq \frac{n}{n+2}C_nV^{-2/n}k^{1+2/n},
\end{equation}
and the monotonicity of the spectrum yields the individual consequence
\begin{equation}\label{eq:li-yau-individual}
 \lambda_k\geq \frac{n}{n+2}C_nV^{-2/n}k^{2/n}.
\end{equation}
Our starting point is to revisit the individual bound
\cref{eq:li-yau-individual} derived from the Li--Yau sum
inequality \cref{eq:li-yau-sum}. Since P\'olya's conjecture
asserts a lower bound for every eigenvalue, improving this
individual estimate is a natural step toward the conjecture.
Although the leading coefficient in the sum inequality is
sharp, we show that energy orthogonality yields a strictly
larger coefficient in the individual bound.

\revise{The first contribution is to strengthen the Fourier constraint
in the Li--Yau argument by retaining information from energy
orthogonality.}
The Dirichlet condition makes the gradients
of the eigenfunctions orthogonal after energy normalization. Expressing a
plane wave as the divergence of an explicit vector field then supplements
the constant Fourier cap in the Li--Yau argument by a second cap that
decays as $|\xi|^{-2}$. A bathtub principle for this radial capacity gives
\[
 \lambda_k\geq c_nC_nV^{-2/n}k^{2/n},
 \qquad \frac{n}{n+2}<c_n<1.
\]
The coefficient is explicit through a scalar equation and applies at every
index, without boundary regularity assumptions. The argument is related
to the constant-cap Fourier method underlying the Li--Yau inequality and
to Berezin's Riesz-mean inequality~\cite{Berezin1972}, but it imposes the
additional energy cap before the moment problem is closed.

The second contribution addresses a different loss. The elementary
closure
\[
 \sum_{j=1}^k\lambda_j\leq k\lambda_k
\]
discards the \revise{nonnegative deficit}
$\sum_{j<k}(\lambda_k-\lambda_j)$. This deficit is exactly the first
Riesz mean at $\lambda_k$, and hence the integral of the strict counting
function. Any valid lower bound for that counting function can therefore
be inserted into the same scalar moment problem. After truncation at
$k-1$, the resulting residual is continuous and strictly decreasing, so
its unique positive root is a rigorous lower bound for $\lambda_k$.
We prove precise sign criteria for strict improvement over both the
volume-only baseline and any independent lower bound.

This hybrid construction is abstract; it requires only a measurable
counting lower bound with the correct direction. On bounded Lipschitz
domains, we implement it using the continuous-energy estimates of
Jiang--Lin~\cite{JiangLin2026}, which are obtained from cube counting,
Whitney decompositions, and boundary-layer control. Their estimate is
integrated here to recover part of the Riesz-mean deficit. The resulting
gain depends on the geometry and may vanish at low energies. Accordingly,
we make no claim of a uniformly larger volume-only coefficient or of
automatic dominance over the stronger special-domain estimates in their
work.

The framework complements several established refinements of Li--Yau and
Berezin inequalities. Positive remainders involving inertia, boundary
layers, or domain width appear in the work of Melas~\cite{Melas2003},
Ilyin~\cite{Ilyin2010}, Kova\v{r}\'ik--Vugalter--Weidl~\cite{KVW2009},
and Geisinger--Laptev--Weidl~\cite{GLW2011}. Steinerberger
\cite{Steinerberger2024} obtains spectral-gap and two-point information;
Frank--Larson--Pfeiffer~\cite{FLP2025} derive quantitative
uncertainty-based improvements of semiclassical Riesz means; and
Gan--Jiang--Lin~\cite{GJL2025} prove refinements with consequences for
infinitely many indices satisfying P\'olya's inequality in higher
dimensions. These results retain information different from the radial
capacity considered here, so meaningful comparisons must be made for the
same spectral quantity and under compatible hypotheses.

The paper is organized as follows. \Cref{sec:preliminaries} records the
spectral identities and proves the radial-capacity bathtub principle.
\Cref{sec:envelope} derives the Fourier envelope from energy orthogonality.
\Cref{sec:main-results} treats the cases $n\geq3$ and $n=2$ and computes
the universal coefficient. \Cref{sec:hybrid} develops the counting-function
closure, proves its well-posedness and comparison criteria, and gives the
Lipschitz-domain input. \Cref{sec:comparison} compares the conclusions with
representative geometric, averaged, high-frequency, and special-domain
bounds. \Cref{sec:questions} summarizes the scope and possible extensions,
and Appendix~\ref{sec:disk-example} provides a reproducible numerical
illustration.

\section{Preliminaries}\label{sec:preliminaries}

We begin by recording the spectral and Fourier identities on which the
argument rests. Let $\{\phi_j\}_{j\geq1}$ be a real-valued orthonormal basis of
Dirichlet eigenfunctions in $L^2(\Omega)$. We use complex Hilbert spaces
when pairing these functions with plane waves, with inner product linear
in its first argument. Thus
\begin{equation}\label{eq:eigenproblem}
-\Delta\phi_j=\lambda_j\phi_j,
\qquad \phi_j\in H_0^1(\Omega),
\qquad \int_\Omega\phi_i\phi_j\dd x=\delta_{ij}.
\end{equation}
We use the unitary Fourier transform
\begin{equation}\label{eq:fourier}
\wh u(\xi)=(2\pi)^{-\frac{n}{2}}
\int_{\R^n}u(x)e^{-\mathrm{i}x\cdot\xi}\dd x,
\end{equation}
where every function in $H_0^1(\Omega)$ is extended by zero to $\R^n$. The
Fourier density associated with the first $k$ eigenfunctions is
\begin{equation}\label{eq:density}
F_k(\xi)=\sum_{j=1}^k\abs{\wh\phi_j(\xi)}^2.
\end{equation}
Plancherel's theorem and the eigenvalue equation then give the two basic
identities
\begin{equation}\label{eq:mass-moment}
\int_{\R^n}F_k(\xi)\dd\xi=k,
\qquad
\int_{\R^n}\abs\xi^2F_k(\xi)\dd\xi
=\sum_{j=1}^k\lambda_j.
\end{equation}
Thus the spectral problem has been converted into a constrained moment problem:
the total mass of $F_k$ is known exactly, its second moment is the eigenvalue
sum, and the remaining task is to obtain the sharpest available pointwise
majorant.

Before deriving that majorant, we formulate the rearrangement principle that
will turn it into a lower bound for the second moment. The following statement
is a radial-obstacle version of the standard bathtub principle; see
Lieb and Loss~\cite{LiebLoss2001}.

\begin{lemma}[Bathtub principle under a radial envelope]\label[lemma]{lem:bathtub}
Let $h:[0,\infty)\to[0,\infty)$ be measurable and nonincreasing, and suppose
that $h(\abs\cdot)\in L^1_{\mathrm{loc}}(\R^n)$. Let $0\leq M<\infty$ and
assume that
\[
\int_{\R^n}h(\abs\xi)\dd\xi\geq M.
\]
If a measurable function $f:\R^n\to[0,\infty)$ satisfies
\[
0\leq f(\xi)\leq h(\abs\xi)\quad\text{a.e.},
\qquad
\int_{\R^n}f(\xi)\dd\xi=M,
\]
then there exists $R\in[0,\infty]$ such that
\[
\int_{B_R}h(\abs\xi)\dd\xi=M.
\]
Consequently, if
\[
g_R(\xi):=h(\abs\xi)\ind_{B_R}(\xi),
\]
then
\begin{equation}\label{eq:bathtub}
\int_{\R^n}\abs\xi^2f(\xi)\dd\xi
\geq
\int_{B_R}\abs\xi^2h(\abs\xi)\dd\xi.
\end{equation}
The integrals in \cref{eq:bathtub} are allowed to take the value $+\infty$.
\end{lemma}

\begin{proof}
The proof consists of three steps.

\medskip
\noindent\emph{Step 1: determination of the filling radius.}
Define the cumulative mass
\[
H(r):=\int_{B_r}h(\abs\xi)\dd\xi,
\qquad r\in[0,\infty).
\]
Since $h(\abs\cdot)\in L^1_{\mathrm{loc}}(\R^n)$, one has $H(r)<\infty$
for every finite $r$. Polar coordinates give
\begin{equation}\label{eq:H-polar}
H(r)=n\omega_n\int_0^r h(s)s^{n-1}\dd s.
\end{equation}
It follows that $H$ is absolutely continuous on every finite interval; in
particular, it is continuous and nondecreasing. Moreover, the monotone
convergence theorem yields
\[
\lim_{r\to\infty}H(r)
=\int_{\R^n}h(\abs\xi)\dd\xi\geq M.
\]
If $M=0$, we take $R=0$. Suppose henceforth that $M>0$. If $H(R)=M$ for some
finite $R$, we choose such an $R$. Otherwise, continuity forces
\[
H(r)<M\quad\text{for every finite }r,
\qquad
\lim_{r\to\infty}H(r)=M,
\]
and we set $R=\infty$. In either case,
\begin{equation}\label{eq:equal-mass}
\int_{\R^n}g_R(\xi)\dd\xi
=\int_{B_R}h(\abs\xi)\dd\xi=M
=\int_{\R^n}f(\xi)\dd\xi.
\end{equation}

\medskip
\noindent\emph{Step 2: comparison when $R<\infty$.}
Inside $B_R$, the bound $f\leq h$ implies
\[
f(\xi)-g_R(\xi)=f(\xi)-h(\abs\xi)\leq0,
\qquad
\abs\xi^2-R^2\leq0.
\]
Outside $B_R$, one has $g_R=0$, and hence
\[
f(\xi)-g_R(\xi)=f(\xi)\geq0,
\qquad
\abs\xi^2-R^2\geq0.
\]
Combining the two regions gives the pointwise sign relation
\begin{equation}\label{eq:pointwise-sign}
\bigl(\abs\xi^2-R^2\bigr)
\bigl(f(\xi)-g_R(\xi)\bigr)\geq0
\quad\text{a.e. in }\R^n.
\end{equation}
Furthermore, local integrability of the envelope implies
\[
\int_{\R^n}\abs\xi^2g_R(\xi)\dd\xi
\leq R^2\int_{B_R}h(\abs\xi)\dd\xi=R^2M<\infty.
\]
If the second moment of $f$ is infinite, the desired conclusion is immediate.
Otherwise, both second moments are finite, and integration of
\cref{eq:pointwise-sign}, together with the equal-mass identity
\cref{eq:equal-mass}, yields
\begin{align*}
0
&\leq\int_{\R^n}
\bigl(\abs\xi^2-R^2\bigr)
\bigl(f(\xi)-g_R(\xi)\bigr)\dd\xi\\
&=\int_{\R^n}\abs\xi^2
\bigl(f(\xi)-g_R(\xi)\bigr)\dd\xi
-R^2\int_{\R^n}\bigl(f(\xi)-g_R(\xi)\bigr)\dd\xi\\
&=\int_{\R^n}\abs\xi^2f(\xi)\dd\xi
-\int_{\R^n}\abs\xi^2g_R(\xi)\dd\xi.
\end{align*}
This is precisely \cref{eq:bathtub}.

\medskip
\noindent\emph{Step 3: the case $R=\infty$.}
In this case $g_R(\xi)=h(\abs\xi)$, while \cref{eq:equal-mass} gives
\[
0\leq h(\abs\xi)-f(\xi),
\qquad
\int_{\R^n}\bigl(h(\abs\xi)-f(\xi)\bigr)\dd\xi=0.
\]
It follows that $f(\xi)=h(\abs\xi)$ almost everywhere. Hence
\cref{eq:bathtub} holds with equality, completing the proof.
\end{proof}

\begin{remark}[Partial filling on the boundary sphere]\label{rem:sphere}
For Lebesgue measure, the sphere $\partial B_R$ has zero $n$-dimensional
measure. Changing the value of $f$ on the sphere therefore neither changes nor
adjusts its mass.
\revise{
If the prescribed mass is strictly smaller than the total
capacity, continuity of $H$ guarantees a finite filling radius
with $H(R)=M$, without any boundary filling. If the prescribed
mass equals the total capacity, the filling radius may be
infinite, as allowed in Lemma~\ref{lem:bathtub}.
In more general bathtub problems, partial filling may be needed
when a level set of the cost function has positive measure
with respect to the capacity measure.
}

\end{remark}

\begin{remark}[Equality]\label{rem:equality}
If $R<\infty$, equality in \cref{eq:bathtub} requires
\[
\bigl(\abs\xi^2-R^2\bigr)
\bigl(f(\xi)-g_R(\xi)\bigr)=0
\quad\text{a.e.}
\]
Since $\partial B_R$ has measure zero, this forces
\[
f(\xi)=h(\abs\xi)\quad\text{a.e. in }B_R,
\qquad
f(\xi)=0\quad\text{a.e. in }\R^n\setminus B_R.
\]
Thus, among all densities of mass $M$ lying below the prescribed capacity
$h(\abs\xi)$, the second moment is minimized by filling the available capacity
from the frequency origin outward.
\end{remark}

The lemma isolates the variational part of the argument. What remains is to
identify a useful radial envelope for $F_k$. The usual $L^2$ argument supplies
a constant cap, whereas the energy inner product supplies a second cap that
decays at high frequency. Their minimum is the crucial input in the next
section.

\section{The Frequency Envelope from Energy Orthogonality}
\label{sec:envelope}

We now derive the main Fourier-space estimate. The first half is the familiar
$L^2$ Bessel bound. The second half uses the Dirichlet condition through weak
integration by parts and is responsible for the improvement.

\begin{proposition}[Improved Fourier projection estimate]\label[proposition]{prop:envelope}
For every $k\in\N$ and almost every $\xi\in\R^n$,
\begin{equation}\label{eq:projection-envelope}
F_k(\xi)
\leq \frac{V}{(2\pi)^n}
\min\left\{1,\frac{\lambda_k}{\abs\xi^2}\right\},
\end{equation}
where only the first term is used at $\xi=0$.
\end{proposition}

\begin{proof}
Regard $(2\pi)^{-n/2}e^{\mathrm{i}x\cdot\xi}$ as an element of
$L^2(\Omega)$. Bessel's inequality gives
\begin{equation}\label{eq:l2-bessel}
F_k(\xi)\leq(2\pi)^{-n}\int_\Omega1\dd x
=\frac{V}{(2\pi)^n}.
\end{equation}

To obtain the second bound, we use the weak formulation of
\cref{eq:eigenproblem}. The vector fields
\begin{equation}\label{eq:energy-orthogonal}
e_j=\frac{\nabla\phi_j}{\sqrt{\lambda_j}},
\qquad j\geq1,
\end{equation}
form an orthonormal family in $L^2(\Omega;\mathbb C^n)$. Fix $\xi\neq0$ and define
\[
X_\xi(x)=\frac{\mathrm{i}\xi}{\abs\xi^2}e^{-\mathrm{i}x\cdot\xi};
\qquad
\diver X_\xi=e^{-\mathrm{i}x\cdot\xi}.
\]
Because $\phi_j\in H_0^1(\Omega)$, weak integration by parts gives
\[
\wh\phi_j(\xi)
=-(2\pi)^{-\frac n2}\int_\Omega
\nabla\phi_j(x)\cdot X_\xi(x)\dd x.
\]
Consequently,
\begin{align*}
F_k(\xi)
&=(2\pi)^{-n}\sum_{j=1}^k
\lambda_j\left|\ip{e_j}{\overline{X_\xi}}_{L^2(\Omega;\mathbb C^n)}\right|^2\\
&\leq(2\pi)^{-n}\lambda_k
\sum_{j=1}^k\left|\ip{e_j}{\overline{X_\xi}}\right|^2\\
&\leq(2\pi)^{-n}\lambda_k\norm{X_\xi}_{L^2(\Omega)}^2
=\frac{V}{(2\pi)^n}\frac{\lambda_k}{\abs\xi^2},
\end{align*}
where the last inequality is the vector-valued $L^2$ Bessel inequality.
Combining this estimate with \cref{eq:l2-bessel} proves
\cref{eq:projection-envelope}.
\end{proof}

\begin{remark}\label{rem:source-improvement}
The constant bound \cref{eq:l2-bessel}, together with
\cref{eq:mass-moment}, recovers the classical Li--Yau inequality. The additional
constraint is active in the region
$\abs\xi>\sqrt{\lambda_k}$.
\revise{When the classical filling ball extends beyond
$\abs{\xi}=\sqrt{\lambda_k}$, the additional cap reduces the
available capacity there and forces the
second-moment-minimizing density to occupy a larger ball.}
\end{remark}

We are now in a position to combine the two ingredients. The projection
estimate supplies a radial obstacle depending on $L=\lambda_k$, and
\cref{lem:bathtub} identifies the density below that obstacle with the smallest
possible second moment. The final closure is provided by the elementary upper
bound $\sum_{j=1}^k\lambda_j\leq k\lambda_k$. Because the radial mass of an
$r^{-2}$ tail changes from a power law to a logarithm at $n=2$, the cases
$n\geq3$ and $n=2$ must be treated separately.

\begin{remark}[A zero-mean refinement]\label{rem:zero-mean-refinement}
The energy estimate admits a further refinement that is not used in the
radial calculations below. Since
$e_j=\nabla\phi_j/\sqrt{\lambda_j}$ and $\phi_j\in H_0^1(\Omega)$,
approximation by compactly supported smooth functions gives
\[
 \int_\Omega e_j(x)\dd x=0.
\]
For $\xi\ne0$, define
\[
 X_\xi(x)=\frac{\mathrm{i}\xi}{|\xi|^2}e^{-\mathrm{i}x\cdot\xi},
 \qquad
 Z_\xi=X_\xi-\frac1V\int_\Omega X_\xi\dd x,
 \qquad
 m_\Omega(\xi)=\int_\Omega e^{-\mathrm{i}x\cdot\xi}\dd x.
\]
Weak integration by parts and the zero-mean identity give
\[
 \wh\phi_j(\xi)
 =-(2\pi)^{-n/2}\sqrt{\lambda_j}\int_\Omega e_j\cdot X_\xi\dd x
 =-(2\pi)^{-n/2}\sqrt{\lambda_j}\int_\Omega e_j\cdot Z_\xi\dd x.
\]
The last integral is
$\langle e_j,\overline{Z_\xi}\rangle_{L^2(\Omega;\mathbb C^n)}$.
The vector-valued Bessel inequality therefore yields
\begin{align*}
 F_k(\xi)
 &\leq(2\pi)^{-n}\lambda_k
       \sum_{j=1}^k
       \left|\left\langle e_j,\overline{Z_\xi}\right\rangle\right|^2\\
 &\leq(2\pi)^{-n}\lambda_k\|Z_\xi\|_{L^2(\Omega)}^2\\
 &=\frac{V}{(2\pi)^n}\frac{\lambda_k}{|\xi|^2}
   \left(1-\frac{|m_\Omega(\xi)|^2}{V^2}\right).
\end{align*}
Consequently, \cref{eq:projection-envelope} can be sharpened to
\begin{equation}\label{eq:projection-envelope-2}
 F_k(\xi)\leq \frac{V}{(2\pi)^n}
 \min\left\{1,\frac{\lambda_k}{|\xi|^2}
 \left(1-\frac{|m_\Omega(\xi)|^2}{V^2}\right)\right\}.
\end{equation}
This capacity is generally nonradial. Exploiting it fully requires a
corresponding nonradial moment problem, so \cref{eq:projection-envelope-2}
is recorded here as a route to a further refinement rather than used in
the volume-only and hybrid theorems proved below.
\end{remark}

\section{Volume-Only Eigenvalue Bounds}\label{sec:main-results}

We first close the moment estimate using only
$\sum_{j\leq k}\lambda_j\leq k\lambda_k$. This yields an explicit bound
on every open set of finite positive measure. Set
\begin{equation}\label{eq:alpha}
\alpha_n=\frac{\omega_nV}{(2\pi)^n}.
\end{equation}

\subsection{Higher dimensions}

For $n\geq3$, let $x_n>1$ be the unique solution in $(1,\infty)$ of
\begin{equation}\label{eq:xn}
(n-2)x^n-nx^{n-2}+\frac{8}{n+2}=0,
\end{equation}
and define
\begin{equation}\label{eq:Kn-cn}
K_n=\frac{nx_n^{n-2}-2}{n-2},
\qquad
c_n=K_n^{-\frac{2}{n}}.
\end{equation}
These constants arise naturally when the lower bound for the Fourier second
moment is matched with its spectral upper bound. The resulting estimate is the
following.

\begin{theorem}[Universal improvement in higher dimensions]
\label{thm:main-high}
Let $\Omega\subset\R^n$ be an open set of finite positive measure and suppose that
$n\geq3$. Then, for every $k\geq1$,
\begin{equation}\label{eq:main-high}
\lambda_k\geq c_nC_nV^{-\frac{2}{n}}k^{\frac{2}{n}}.
\end{equation}
Moreover,
\begin{equation}\label{eq:strict-comparison}
\frac{n}{n+2}<c_n<1.
\end{equation}
\end{theorem}

\begin{proof}
Write $L=\lambda_k$ and introduce the dimensionless parameter
\begin{equation}\label{eq:rho}
\rho=\frac{k}{\alpha_nL^{\frac n2}}.
\end{equation}
If $\rho\leq1$, then rearrangement of this inequality immediately gives the
P\'olya bound \cref{eq:polya}, which is stronger than
\cref{eq:main-high}. We may therefore assume that $\rho>1$.

In terms of $L$, the envelope in \cref{prop:envelope} is
\[
h_L(r)=\frac{V}{(2\pi)^n}\min\left\{1,\frac{L}{r^2}\right\}.
\]
By \cref{lem:bathtub}, the admissible density of mass $k$ with the smallest
second moment fills this envelope on a ball $B_R$. Put $R=x\sqrt L$.
The condition $\rho>1$ forces $x>1$, so the filling ball extends into the
decaying part of the envelope. Direct radial integration then yields
\begin{align}
k&=\alpha_nL^{\frac n2}
\frac{nx^{n-2}-2}{n-2},\label{eq:mass-high}\\
\sum_{j=1}^k\lambda_j
&\geq\alpha_nL^{1+\frac n2}
\left(x^n-\frac{2}{n+2}\right).
\label{eq:moment-high}
\end{align}
On the other hand, monotonicity of the eigenvalues gives the complementary
upper bound
\begin{equation}\label{eq:upper-moment}
\sum_{j=1}^k\lambda_j\leq k\lambda_k=kL.
\end{equation}
Substituting \cref{eq:mass-high} into
\cref{eq:moment-high,eq:upper-moment} gives
\begin{equation}\label{eq:root-inequality}
(n-2)x^n-nx^{n-2}+\frac{8}{n+2}\leq0.
\end{equation}
The derivative of the left-hand side is
\[
n(n-2)x^{n-3}(x^2-1)>0
\qquad (x>1).
\]
The same function is negative at $x=1$ and tends to $+\infty$ as
$x\to\infty$. Hence it has exactly one zero $x_n$ in $(1,\infty)$, and
\cref{eq:root-inequality} implies $x\leq x_n$. Using
\cref{eq:mass-high}, we therefore obtain
\[
\rho=\frac{nx^{n-2}-2}{n-2}\leq K_n.
\]
Returning to \cref{eq:rho} and using \cref{eq:weyl-constant}, we conclude that
\[
L\geq\left(\frac{k}{\alpha_nK_n}\right)^{\frac2n}
=c_nC_nV^{-\frac2n}k^{\frac2n}.
\]

It remains to locate the coefficient relative to the Li--Yau and P\'olya
constants. Since $K_n>1$, one immediately has $c_n<1$. To see that the
improvement over Li--Yau is strict without resorting to a lengthy algebraic
comparison, suppose that $L$ equals the right-hand side of
\cref{eq:li-yau-individual}. For the classical constant envelope, the bathtub
lower bound then equals $kL$. At that value of $L$,
\[
\rho=\left(\frac{n+2}{n}\right)^{\frac{n}{2}}>1,
\]
so the minimizing ball necessarily reaches the region
$\abs\xi>\sqrt L$. There the new envelope is strictly below the classical
constant one. Consequently, the smallest second moment of a density with mass
$k$ is strictly greater than $kL$, contradicting
\cref{eq:upper-moment}. Continuity now gives $c_n>n/(n+2)$ and completes the
proof.
\end{proof}

\subsection{The critical two-dimensional case}

When $n=2$, the radial mass of the $r^{-2}$ tail grows logarithmically, so the
power-law calculation above has to be replaced by a separate computation. Let
$y_2>1$ be the unique solution in $(1,\infty)$ of
\begin{equation}\label{eq:y2}
y-\log y=\frac32,
\end{equation}
and define
\begin{equation}\label{eq:c2}
K_2=1+\log y_2=y_2-\frac12,
\qquad
c_2=K_2^{-1}.
\end{equation}

\begin{theorem}[Universal improvement in dimension two]
\label{thm:main-2d}
Let $\Omega\subset\R^2$ be an open set of finite positive measure. Then, for every
$k\geq1$,
\begin{equation}\label{eq:main-2d}
\lambda_k\geq c_2\frac{4\pi k}{V},
\qquad c_2=0.5383068077\ldots.
\end{equation}
In particular, the coefficient improves the individual Li--Yau bound
$\lambda_k\geq2\pi k/V$ by approximately $7.66\%$.
\end{theorem}

\begin{proof}
Again write $L=\lambda_k$, and note that $\alpha_2=V/(4\pi)$. If
$k\leq\alpha_2L$, then the P\'olya bound already follows. Otherwise, write the
saturation radius as $R^2=yL$ with $y>1$. Radial integration gives
\begin{align}
k&=\alpha_2L(1+\log y),\label{eq:mass-2d}\\
\sum_{j=1}^k\lambda_j
&\geq\alpha_2L^2\left(y-\frac12\right).
\label{eq:moment-2d}
\end{align}
Combining \cref{eq:upper-moment,eq:mass-2d,eq:moment-2d}, we obtain
\[
y-\frac12\leq1+\log y,
\qquad\text{or equivalently}\qquad
y-\log y\leq\frac32.
\]
The function $y-\log y$ is strictly increasing for $y>1$. Hence
$y\leq y_2$, and therefore
\[
\frac{k}{\alpha_2L}=1+\log y\leq K_2.
\]
Rearranging proves \cref{eq:main-2d}. The stated numerical value follows from
\cref{eq:y2,eq:c2}.
\end{proof}

\subsection{Numerical comparison of the constants}

To illustrate the size of the improvement, \cref{tab:constants} lists the
coefficients in several dimensions. The relative gain is defined as
\[
\left(\frac{c_n}{n/(n+2)}-1\right)\times100\%.
\]
The displayed gains decrease across the sampled dimensions. Strict
positivity in every dimension $n\geq2$ follows from the preceding theorems.

\begin{table}[ht]
\centering
\caption{Comparison between the coefficient obtained here and the Li--Yau
coefficient.}\label{tab:constants}
\begin{tabular}{@{}cccc@{}}
\toprule
$n$ & Li--Yau coefficient $n/(n+2)$ & Present coefficient $c_n$
& Relative gain \\
\midrule
2  & 0.500000 & 0.538307 & 7.66\% \\
3  & 0.600000 & 0.622248 & 3.71\% \\
4  & 0.666667 & 0.681250 & 2.19\% \\
5  & 0.714286 & 0.724595 & 1.44\% \\
6  & 0.750000 & 0.757678 & 1.02\% \\
8  & 0.800000 & 0.804735 & 0.59\% \\
10 & 0.833333 & 0.836544 & 0.39\% \\
\bottomrule
\end{tabular}
\end{table}

\section{A Hybrid Refinement on Lipschitz Domains}\label{sec:hybrid}

The volume-only bound is obtained by comparing the optimal envelope
moment with $k\lambda_k$. This comparison discards the \revise{nonnegative deficits}
$\lambda_k-\lambda_j$ for $j<k$. We now retain a quantitatively controlled
part of their sum. The resulting refinement combines Fourier and geometric
information through two estimates for the same spectral quantity.

Although the explicit application in this section combines our
framework with Jiang--Lin's counting estimates on bounded
Lipschitz domains \cite{JiangLin2026}, the abstract analysis in the first three
subsections applies to any open set $\Omega\subset\R^n$ of
finite positive measure, with $n\geq2$.
Retain $V$, $C_n$, and $\alpha_n$ from
\cref{eq:weyl-constant,eq:alpha}, and set
\[
 p=\frac n2,\qquad W_k=C_nV^{-\frac{2}{n}}k^{\frac{2}{n}}
 =\left(\frac{k}{\alpha_n}\right)^{\frac{2}{n}},\qquad
 S_k=\sum_{j=1}^k\lambda_j.
\]
We use a strict counting function and the first Riesz mean:
\begin{equation}\label{eq:hybrid-count-riesz}
 N_\Omega(t)=\#\{j:\lambda_j<t\},\qquad
 R_1(L)=\sum_{j\geq1}(L-\lambda_j)_+.
\end{equation}

\subsection{The exact spectral deficit}

The relevant correction is an integral of the counting function.
Because this correction is subtracted from $k\lambda_k$, a counting
\emph{lower} bound supplies the required direction of comparison.

\begin{lemma}[Integrated counting and multiplicities]\label[lemma]{lem:hybrid-deficit}
For every $L>0$,
\begin{equation}\label{eq:hybrid-riesz-integral}
 R_1(L)=\int_0^L N_\Omega(t)\dd t.
\end{equation}
In particular, for every $k\geq1$,
\begin{equation}\label{eq:hybrid-exact-deficit}
 S_k=k\lambda_k-R_1(\lambda_k).
\end{equation}
Let $\nu:(0,\infty)\to[0,\infty)$ be measurable with
$\nu(t)\leq N_\Omega(t)$ almost everywhere, and define
\begin{equation}\label{eq:hybrid-capped-count}
 q_k(t)=\min\{k-1,\nu(t)\},\qquad
 J_k(L)=\int_0^L q_k(t)\dd t.
\end{equation}
Then $J_k$ is nondecreasing and $(k-1)$-Lipschitz, and
\begin{equation}\label{eq:hybrid-upper-sum}
 S_k\leq k\lambda_k-J_k(\lambda_k).
\end{equation}
\end{lemma}

\begin{proof}
For each $j$, $(L-\lambda_j)_+=\int_0^L\one_{\{\lambda_j<t\}}\dd t$.
Tonelli's theorem proves \cref{eq:hybrid-riesz-integral}. If $L=\lambda_k$,
all terms with $j>k$ vanish, while any terms equal to $L$ also contribute
zero. Thus $R_1(\lambda_k)=\sum_{j=1}^k(\lambda_k-\lambda_j)$, including
when $\lambda_k$ has multiplicity greater than one. For $0<t\leq\lambda_k$,
there are at most $k-1$ eigenvalues strictly below $t$. Consequently,
$q_k(t)\leq N_\Omega(t)$ and
$J_k(\lambda_k)\leq R_1(\lambda_k)$. This gives
\cref{eq:hybrid-upper-sum}. Finally, $0\leq q_k\leq k-1$ proves all
regularity assertions.
\end{proof}

The truncation at $k-1$ does not remove any usable information below the
true $\lambda_k$. Its additional benefit is that it controls the integral
for arbitrary trial values $L$, making the root problem globally monotone.
In contrast, neither $N_\Omega(\lambda_k)=k$ nor a substitution of
eigenvalue lower bounds into spectral deficits would be justified.

\subsection{The envelope moment for arbitrary trial energies}

To use \cref{eq:hybrid-upper-sum}, we need the bathtub value as a function
of an arbitrary positive $L$, rather than only at $L=\lambda_k$. Define
\begin{equation}\label{eq:hybrid-moment-definition}
 \mathcal M_k(L)=\inf\left\{\int_{\R^n}|\xi|^2f(\xi)\dd\xi:
 0\leq f(\xi)\leq\frac{V}{(2\pi)^n}
 \min\{1,L/|\xi|^2\},\quad\int_{\R^n}f=k\right\}.
\end{equation}
The value of the envelope at the origin is interpreted as $V/(2\pi)^n$.
For $n\geq2$, its total mass is infinite, whereas its mass on every finite
ball is finite. Thus the filling radius for every finite $k$ exists and
is finite.

\begin{lemma}[Explicit moment function]\label[lemma]{lem:hybrid-moment}
Write $\rho=k/(\alpha_nL^p)$. If $\rho\leq1$, then
\begin{equation}\label{eq:hybrid-moment-low}
 \mathcal M_k(L)=\frac{n}{n+2}kW_k.
\end{equation}
If $\rho>1$ and $n\geq3$, then
\begin{equation}\label{eq:hybrid-moment-high}
 x=\left(\frac{(n-2)\rho+2}{n}\right)^{1/(n-2)},\qquad
 \mathcal M_k(L)=\alpha_nL^{p+1}\left(x^n-\frac2{n+2}\right).
\end{equation}
In dimension two, the full formula is
\begin{equation}\label{eq:hybrid-moment-planar}
 \mathcal M_k(L)=
 \begin{cases}
 k^2/(2\alpha_2),&k\leq\alpha_2L,\\[2pt]
 \alpha_2L^2\bigl(e^{k/(\alpha_2L)-1}-\tfrac12\bigr),
     &k>\alpha_2L.
 \end{cases}
\end{equation}
The function $\mathcal M_k$ is continuous and nonincreasing on $(0,\infty)$,
tends to infinity as $L\downarrow0$, and is constant for $L\geq W_k$.
\end{lemma}

\begin{proof}
By \cref{lem:bathtub}, the minimizer fills the envelope on a ball. If
$\rho\leq1$, its radius is $\sqrt{W_k}\leq\sqrt L$, so the obstacle is
constant throughout the filling ball. This gives
\cref{eq:hybrid-moment-low}. For $\rho>1$, write the radius as $x\sqrt L$.
The mass and moment integrals are exactly those in
\cref{eq:mass-high,eq:moment-high}; solving the mass equation proves
\cref{eq:hybrid-moment-high}. In dimension two,
$\rho=1+\log y$ with $y=R^2/L$, so $y=e^{\rho-1}$ and
\cref{eq:mass-2d,eq:moment-2d} give \cref{eq:hybrid-moment-planar}.

The formulas agree at $\rho=1$, proving continuity there and hence
everywhere. If $L$ increases, the admissible set in
\cref{eq:hybrid-moment-definition} enlarges, so the infimum cannot
increase. For $n\geq3$, the high-frequency formula grows like a positive
constant times $L^{-2/(n-2)}$ as $L\downarrow0$; for $n=2$, the
exponential in \cref{eq:hybrid-moment-planar} gives the same divergence.
Finally, $L\geq W_k$ is equivalent to $\rho\leq1$.
\end{proof}

\subsection{A unique-root lower bound and its comparison criteria}

The next theorem separates the analytic construction from the choice
of geometric input. It applies to any admissible $\nu$ and does not
require knowledge of the eigenfunctions.

\begin{theorem}[Hybrid spectral closure]\label{thm:hybrid-root}
Let $\nu$ satisfy the assumptions of \cref{lem:hybrid-deficit}, and let
$J_k$ be defined by \cref{eq:hybrid-capped-count}. The equation
\begin{equation}\label{eq:hybrid-root}
 \Phi_k(L):=\mathcal M_k(L)+J_k(L)-kL=0
\end{equation}
has a unique positive solution $L_k^{\mathrm{hyb}}$. For every $k\geq1$,
\begin{equation}\label{eq:hybrid-lower-bound}
 \lambda_k\geq L_k^{\mathrm{hyb}}\geq L_k^{(0)},\qquad
 L_k^{(0)}:=c_nW_k.
\end{equation}
The second inequality is strict if and only if
\begin{equation}\label{eq:hybrid-strict-original}
 J_k(L_k^{(0)})>0.
\end{equation}
More generally, for any other valid lower bound $B_k\leq\lambda_k$, put
$L_*:=\max\{L_k^{(0)},B_k\}$. Then
\begin{equation}\label{eq:hybrid-strict-maximum}
 L_k^{\mathrm{hyb}}>L_*
 \quad\Longleftrightarrow\quad
 \mathcal M_k(L_*)+J_k(L_*)>kL_*.
\end{equation}
\end{theorem}

\begin{proof}
For $L_2>L_1>0$, \cref{lem:hybrid-deficit,lem:hybrid-moment} imply
\begin{align}
 \Phi_k(L_2)-\Phi_k(L_1)
 &\leq (k-1)(L_2-L_1)-k(L_2-L_1)\notag\\
 & =-(L_2-L_1)<0.\label{eq:hybrid-strong-monotonicity}
\end{align}
Thus $\Phi_k$ is continuous and strictly decreasing. As $L\downarrow0$,
$\Phi_k(L)\to+\infty$. For $L\geq W_k$,
\[
 \Phi_k(L)\leq\frac{n}{n+2}kW_k-L\longrightarrow-\infty.
\]
The intermediate value theorem gives a unique positive zero.

At the true eigenvalue, \cref{prop:envelope,eq:mass-moment} and the
definition of $\mathcal M_k$ give
$\mathcal M_k(\lambda_k)\leq S_k$. Combining this with
\cref{eq:hybrid-upper-sum} yields $\Phi_k(\lambda_k)\leq0$ and hence
$\lambda_k\geq L_k^{\mathrm{hyb}}$.

The constants in \cref{eq:Kn-cn,eq:c2} were obtained by matching the
same bathtub value to $kL$. Equivalently, their defining equations and
\cref{lem:hybrid-moment} give
$\mathcal M_k(L_k^{(0)})=kL_k^{(0)}$. Therefore
$\Phi_k(L_k^{(0)})=J_k(L_k^{(0)})\geq0$.
Strict monotonicity proves both the second inequality in
\cref{eq:hybrid-lower-bound} and its equality criterion.
Applying the same sign test at $L_*$ proves
\cref{eq:hybrid-strict-maximum}.
\end{proof}

For $k=1$ the truncation makes $J_1=0$, so this particular refinement
does not improve the first eigenvalue. For $k\geq2$, strict improvement
requires usable lower-counting information on a set of positive measure
below $L_k^{(0)}$. Neither a spectral gap at a single point nor a formal
substitution of lower bounds for earlier eigenvalues establishes this
condition.

\begin{corollary}[Monotonicity in the geometric input]\label[corollary]{cor:hybrid-input}
If $0\leq\nu_1\leq\nu_2\leq N_\Omega$ almost everywhere, then
$L_k^{\mathrm{hyb}}[\nu_1]\leq L_k^{\mathrm{hyb}}[\nu_2]$.
\end{corollary}
\begin{proof}
The corresponding integrals satisfy $J_{k,1}\leq J_{k,2}$, so
$\Phi_{k,2}\geq\Phi_{k,1}$. Compare their unique zeros using
\cref{eq:hybrid-strong-monotonicity}.
\end{proof}

\begin{corollary}[A strict upper bound for the hybrid root]
\label[corollary]{cor:hybrid-weyl-upper}
Under the assumptions of \cref{thm:hybrid-root}, the hybrid root satisfies
\[
 L_k^{\mathrm{hyb}}<W_k,\qquad k\geq1.
\]
\end{corollary}

\begin{proof}
Since $q_k\leq N_\Omega$ almost everywhere,
\cref{eq:hybrid-riesz-integral} gives
\[
 J_k(L)\leq R_1(L),\qquad L>0.
\]
Fix $L>0$, set $m=N_\Omega(L)$, and let
$F_m=\sum_{j=1}^m|\wh\varphi_j|^2$, with $F_0=0$.
By \cref{eq:mass-moment,eq:l2-bessel},
\[
 R_1(L)=\int_{\R^n}(L-|\xi|^2)F_m(\xi)\dd\xi,
 \qquad 0\leq F_m\leq V/(2\pi)^n.
\]
Discarding the nonpositive contribution outside $B_{\sqrt L}(0)$ and
integrating in polar coordinates, we obtain
\begin{align*}
 R_1(L)
 &\leq\int_{|\xi|<\sqrt L}(L-|\xi|^2)F_m(\xi)\dd\xi\\
 &\leq V/(2\pi)^n\int_{|\xi|<\sqrt L}(L-|\xi|^2)\dd\xi\\
 &=V/(2\pi)^n n\omega_n\int_0^{\sqrt L}(L-r^2)r^{n-1}\dd r
 =\frac{2\alpha_n}{n+2}L^{p+1}.
\end{align*}
The resulting bound is strict. This is immediate if $m=0$. 
\revise{Otherwise, equality $F_m(0)=V/(2\pi)^n$ in Bessel's inequality would imply
$1=\sum_{j=1}^m a_j\varphi_j$ in $L^2(\Omega)$ with $a_j=\ip{\varphi_j}{1}_{L^2(\Omega)}$. 
Since the finite linear combination equals the constant function in \(L^2(\Omega)\), 
its weak gradient vanishes. Energy orthogonality then gives}
\[
 0=\int_\Omega\left|\nabla\sum_{j=1}^m a_j\varphi_j\right|^2\dd x
   =\sum_{j=1}^m\lambda_j|a_j|^2,
\]
which forces every $a_j$ to vanish, a contradiction. Thus $F_m(0)<V/(2\pi)^n$.
Moreover, $\|\varphi_j\|_{L^1(\Omega)}\leq V^{1/2}$, so $F_m$ is continuous.
Consequently $F_m<V/(2\pi)^n$ on a neighborhood of the origin contained in
$B_{\sqrt L}(0)$, making the second integral inequality strict. Hence
\[
 J_k(L)\leq R_1(L)<\frac{2\alpha_n}{n+2}L^{p+1},\qquad L>0.
\]
Using $\alpha_nW_k^p=k$ and \cref{eq:hybrid-moment-low}, we conclude that
\[
 \Phi_k(W_k)=J_k(W_k)-\frac{2}{n+2}kW_k<0.
\]
Since $\Phi_k$ is strictly decreasing and
$\Phi_k(L_k^{\mathrm{hyb}})=0$, this proves $L_k^{\mathrm{hyb}}<W_k$.
\end{proof}

\subsection{An explicit Lipschitz-domain input from Jiang--Lin}
For the explicit construction in this subsection, assume that
$\Omega$ is a bounded Lipschitz domain. 
It remains to supply a lower count that is valid at continuous energy
parameters, not just a relation between $k$ and $\lambda_k$. This is
provided by the proof of Jiang--Lin~\cite[Section 4, Eq.~(4.1)]{JiangLin2026}.
Write $P=|\partial\Omega|$ and $r=r_{\mathrm{in}}(\Omega)>0$.
Choose a valid boundary-layer constant $\mathcal C_\Omega\geq1$ such that
\begin{equation}\label{eq:hybrid-boundary-layer}
 |\{x\in\Omega:\operatorname{dist}(x,\partial\Omega)<s\}|
 \leq\mathcal C_\Omega P s\qquad(s>0).
\end{equation}
Such a constant exists for a bounded Lipschitz domain: a local tubular
estimate handles small $s$, and enlargement of the constant handles the
remaining scales using the bound $V$. One may enlarge the constant used
in \cite{JiangLin2026} if necessary; a smallest constant is not required.

Their Whitney-cube argument gives, for $t>0$ and $0<d<r$,
\begin{equation}\label{eq:hybrid-jl-input}
 N_\Omega(t)\geq\alpha_n t^p
 -\frac{\alpha_n\mathcal C_\Omega P}{V}t^{p-1/2}
 \left(d\sqrt t+5n^2\pi\log_2\frac{10r}{d}\right).
\end{equation}
For clarity, the mechanism behind this input is as follows. On a cube of
side length $\ell$, lattice counting yields the lower estimate
\[
 N_Q(t)\geq\frac{\omega_n}{(2\pi)^n}
 \bigl(\ell^nt^p-n^{3/2}\pi\ell^{n-1}t^{p-1/2}\bigr).
\]
Zero extension of the cube trial spaces allows the counts to be summed.
The discarded Whitney cubes lie in a boundary layer, whose missing
volume is controlled by \cref{eq:hybrid-boundary-layer}. At each dyadic
scale the sum of the cube errors is bounded by the same multiple of
$Pt^{p-1/2}$; summing the scales introduces the logarithm. The explicit
constants and cutoff accounting are those in the cited proof.

A convenient all-energy choice is
\begin{align}
 d(t)&=\min\{r/2,t^{-1/2}\},\notag\\
 \nu_{\mathrm{JL}}(t)
 &=\left[\alpha_n t^p
 -\frac{\alpha_n\mathcal C_\Omega P}{V}t^{p-1/2}
 \left(d(t)\sqrt t+5n^2\pi\log_2\frac{10r}{d(t)}\right)\right]_+.
 \label{eq:hybrid-explicit-count}
\end{align}
Set $\nu_{\mathrm{JL}}(0)=0$. This function is measurable, locally
bounded, and satisfies $0\leq\nu_{\mathrm{JL}}\leq N_\Omega$.
It can therefore be inserted directly into \cref{thm:hybrid-root}.
The parameter $d$ is a geometric truncation length; it is not the
$\varepsilon$ in the epsilon-loss formulation of P\'olya's conjecture.

This construction yields a valid bound at every index, although the
strict-improvement condition may fail at low energies where the positive
part in \cref{eq:hybrid-explicit-count} is zero. More efficient geometric
decompositions can be used instead. For example, disjoint interior cubes
give a lower count by summing their explicit spectra, and independent
certified eigenvalue upper bounds give another lower count. Taking the
maximum of such lower counts is admissible, and
\cref{cor:hybrid-input} shows that it cannot weaken the hybrid bound.

\subsection{Implementation, comparison, and scaling}

The theorem reduces the calculation to integration and a scalar monotone
root. By \cref{thm:hybrid-root,cor:hybrid-weyl-upper}, a guaranteed
bracket is
\begin{equation}\label{eq:hybrid-bracket}
L_k^{(0)}\leq L_k^{\mathrm{hyb}}<W_k.
\end{equation}
The proof of \cref{thm:hybrid-root} gives
$\Phi_k(L_k^{(0)})\geq0$, while the proof of
\cref{cor:hybrid-weyl-upper} gives $\Phi_k(W_k)<0$.
Thus bisection on $[L_k^{(0)},W_k]$ is sufficient. For a certified numerical
lower bound, evaluate a lower enclosure of $J_k$ and a controlled
enclosure of $\mathcal M_k$, and retain a trial value only when the lower
enclosure of $\Phi_k$ is nonnegative. An ordinary floating-point root is
an approximation, not automatically a certified spectral bound.

In \cref{eq:hybrid-strict-maximum}, $B_k$ may be the lower bound obtained
from the upper-counting side of Jiang--Lin's Theorem 1.5, or their
epsilon-loss estimate after its threshold has been verified. The hybrid
construction uses the other side of their counting theorem, so a gain
over the maximum of the two separate lower bounds is possible, but must
be checked by \cref{eq:hybrid-strict-maximum}.
In particular, \cref{cor:hybrid-weyl-upper} implies that if
$B_k\geq W_k$, then $L_k^{\mathrm{hyb}}<W_k\leq B_k$, so no such
improvement is possible.
No universal dominance over their complete geometric estimates is asserted.

The construction respects dilation. If $s>0$ and a lower count is
transported by $\nu_{s\Omega}(t)=\nu_\Omega(s^2t)$, then
\[
 \mathcal M_{k,s\Omega}(L)=s^{-2}\mathcal M_{k,\Omega}(s^2L),\qquad
 J_{k,s\Omega}(L)=s^{-2}J_{k,\Omega}(s^2L).
\]
Consequently $L_{k,s\Omega}^{\mathrm{hyb}}=s^{-2}L_{k,\Omega}^{\mathrm{hyb}}$.
The explicit choice \cref{eq:hybrid-explicit-count} has precisely this
covariance when the same dimensionless boundary-layer constant is used.
Thus the additional information is genuinely geometric, not a change of
physical scaling. The argument by itself does not improve the logarithmic
order of the Jiang--Lin remainder. 
Moreover, \cref{cor:hybrid-weyl-upper} shows that
$L_k^{\mathrm{hyb}}<W_k$ for every admissible counting input. Thus improving
$\nu$ alone, with the moment function $\mathcal M_k$ fixed, cannot make the
hybrid bound reach the P\'olya value $W_k$. 

The numerical illustration in Appendix~\ref{sec:disk-example} exhibits
how the two inputs interact. It also distinguishes approximate root
values from certified spectral bounds and from sharper estimates
available on special domains.

The hybrid result therefore complements, rather than replaces, the
volume-only estimate of \cref{sec:main-results}. The next section
compares that baseline with existing work and indicates how independent
information can be used alongside the hybrid construction.

\section{Relation to Existing Results}\label{sec:comparison}

We now compare the two levels of the preceding construction with
representative results in the literature. Use the common normalization
\[
 W_k=C_nV^{-2/n}k^{2/n},\qquad
 a_n=\frac{n}{n+2},\qquad
 S_k=\sum_{j=1}^k\lambda_j.
\]
The universal conclusion is $\lambda_k\geq c_nW_k$, whereas
\cref{thm:hybrid-root} provides an additional geometry-dependent lower
bound on bounded Lipschitz domains. Unless explicitly stated otherwise,
the comparisons below concern the universal baseline.

Three distinctions are essential. First, an estimate for a single
eigenvalue, one for an eigenvalue sum, and one for a Riesz mean are not
interchangeable: an improvement for one quantity need not transfer with
the same coefficient to another. Second, every geometric comparison is
made only under the hypotheses of the cited theorem, including finiteness
of the moment of inertia when it appears. Third, a statement valid for
every index is logically different from a coefficient-one result on a
special domain, above a frequency threshold, or along a subsequence.
These distinctions prevent numerical comparisons from overstating the
scope of the present bounds.

\subsection{The original Li--Yau theorem and the meaning of the improvement}

Li--Yau~\cite[Theorem 1]{LiYau1983} establishes
$S_k\geq a_nkW_k$; using $S_k\leq k\lambda_k$ gives
$\lambda_k\geq a_nW_k$. Our comparison is therefore
\begin{equation}\label{eq:three-levels}
 a_nW_k<c_nW_k<W_k.
\end{equation}
The relative gain over this particular individual bound is
$c_n/a_n-1$, independently of the volume and index. In two dimensions
it equals $2c_2-1=0.0766136154\ldots$.

There is no conflict with the optimal leading constant in the sum inequality.
Indeed, summing our individual estimates yields only
\[
 S_k\geq c_nC_nV^{-\frac{2}{n}}\sum_{j=1}^k j^{\frac{2}{n}}
 =\bigl(c_na_n+o(1)\bigr)C_nV^{-\frac{2}{n}}k^{1+\frac{2}{n}}.
\]
Since $c_n<1$, this asymptotic coefficient is smaller than the Li--Yau
coefficient $a_n$. Thus our individual result does not strengthen the
sharp leading term for sums. Nor does it improve every known estimate for
individual eigenvalues: on domains where P\'olya's inequality is known,
the coefficient one is stronger. At low indices, isoperimetric eigenvalue
bounds can also be substantially stronger.

The analytic difference is equally precise. Li--Yau uses the cap
$F_k\leq V/(2\pi)^n$, whereas \cref{prop:envelope} adds the cap
$F_k\leq V\lambda_k/((2\pi)^n|\xi|^2)$. The improvement comes from imposing
both restrictions before minimizing the moment.

\subsection{Melas and Ilyin: inertia corrections and explicit comparisons}

For bounded domains, Melas~\cite{Melas2003} adds a positive term linear in
$k$ to the sum estimate. One valid explicit version, also recorded in
Ilyin~\cite[Eq.~(1.2)]{Ilyin2010}, is
\begin{equation}\label{eq:melas-comparison}
 S_k\geq a_nkW_k+\mu_n\frac{V}{I(\Omega)}k,
 \qquad \mu_n=\frac{1}{24(n+2)},\qquad
 I(\Omega)=\inf_{a\in\R^n}\int_\Omega|x-a|^2\dd x.
\end{equation}
Consequently, the direct individual corollary is
$\lambda_k\geq a_nW_k+\mu_n V/I(\Omega)$.
Our right-hand side is at least as large precisely when
\begin{equation}\label{eq:melas-threshold}
 k\geq
 \left(\frac{\mu_n V^{1+\frac{2}{n}}}
 {(c_n-a_n)C_n I(\Omega)}\right)^{\frac{n}{2}}.
\end{equation}
This comparison concerns the displayed individual corollary, not the
full information in \cref{eq:melas-comparison}.

Ilyin~\cite{Ilyin2010} refines the Fourier minimization by imposing a
gradient constraint. After translating the domain to its centroid, the
additional condition takes the form
\[
 |\nabla F_k|\leq 2(2\pi)^{-n}\sqrt{VI(\Omega)}.
\]
It controls the slope of the density, whereas our constraint controls its
height as a function of frequency. These are distinct restrictions.
For example, Ilyin's planar estimate~\cite[Theorem 4.1]{Ilyin2010} gives
\begin{equation}\label{eq:ilyin-planar}
 \lambda_k\geq\frac{2\pi k}{V}
 +\frac{1}{24}\frac{V}{I(\Omega)}
       \left(1-\frac{1}{120k}\right)
\end{equation}
after division of the sum bound by $k$.

An explicit geometric calculation makes the planar comparison particularly
transparent. Among sets of area $V$, the disk minimizes the second moment,
so
\[
 I(\Omega)\geq\frac{V^2}{2\pi}.
\]
Indeed, for the disk $B$ centered at the origin with $|B|=V$,
integrating $(|x|^2-R^2)(\one_\Omega-\one_B)\geq0$ proves
$\int_\Omega|x|^2\dd x\geq\int_B|x|^2\dd x=V^2/(2\pi)$;
translation and minimization give the assertion.
The increment over $2\pi k/V$ in \cref{eq:ilyin-planar} is therefore
less than $\pi/(12V)$. Our increment is
$4\pi(c_2-\tfrac12)k/V$, which already exceeds $\pi/(12V)$ at $k=1$.
Thus our planar bound is stronger, for every $k$, than this direct
individual corollary of Ilyin and hence also than the displayed Melas
corollary. This conclusion does not compare the original sum inequalities
or more elaborate uses of them.

\subsection{Boundary corrections and Riesz-mean estimates}
Kova\v{r}\'ik--Vugalter--Weidl~\cite{KVW2009} obtain planar sum estimates
with a positive correction of almost the second asymptotic order under
their boundary hypotheses. Geisinger--Laptev--Weidl~\cite{GLW2011}
derive geometric improvements of Berezin estimates for Riesz means of
order $\sigma\geq3/2$, with a negative remainder of the appropriate
semiclassical order. Under additional geometric assumptions they also
obtain improved individual bounds. These results retain information about
the boundary that the volume-only coefficient $c_n$ cannot encode.

For a fair numerical comparison, a remainder must first be converted to
the same spectral quantity. If a valid sum estimate is
\[
 S_k\geq a_nkW_k+\mathcal R_\Omega(k),
\]
then its direct individual consequence beats $c_nW_k$ exactly when
\begin{equation}\label{eq:remainder-criterion}
 \mathcal R_\Omega(k)> (c_n-a_n)kW_k.
\end{equation}
For a fixed domain, a remainder of size $k^{1+1/n}$ is asymptotically
smaller than the threshold $k^{1+2/n}$ in
\cref{eq:remainder-criterion}. This observation compares direct
individual corollaries only. It does not diminish the sharper information
about sums or boundary asymptotics in those papers.

To discuss more recent estimates, define
\[
 R_1(z)=\sum_j(z-\lambda_j)_+,\qquad
 L_{1,n}^{\rm cl}=\frac{2\omega_n}{(n+2)(2\pi)^n}.
\]
Frank--Larson--Pfeiffer~\cite[Theorem 1]{FLP2025} prove, on open sets
of finite measure, an estimate of the form
\begin{equation}\label{eq:flp-comparison}
 R_1(z)\leq L_{1,n}^{\rm cl}Vz^{1+\frac{n}{2}}
 \left(1-b_n e^{-d_n\sqrt z\,V^{\frac{1}{n}}}\right),
 \qquad b_n,d_n>0.
\end{equation}
The constants depend only on dimension. Their multiplicative gain acts
on a Riesz mean and depends on the energy, unlike our fixed coefficient
for $\lambda_k$. A direct route to comparison is the elementary identity
\begin{equation}\label{eq:riesz-duality}
 S_k=\sup_{z\geq0}\{kz-R_1(z)\}.
\end{equation}
At $z=\lambda_k$ equality holds even with multiplicities; other values
give no larger expression. Thus any verified upper bound $R_1(z)\leq U(z)$
gives $\lambda_k\geq k^{-1}\sup_z\{kz-U(z)\}$.
Numerical dominance requires evaluating that bound with valid constants.
The factor in \cref{eq:flp-comparison} tends to one as $z\to\infty$,
so it should not be interpreted as an energy-independent improvement
of the leading Riesz-mean constant.

\subsection{Steinerberger: spectral gaps and two-point information}

Steinerberger~\cite[Theorem 1]{Steinerberger2024} proves for bounded
domains that
\begin{equation}\label{eq:steiner-gap}
 \lambda_k\geq a_nW_k+\frac{r}{k}(\lambda_k-\lambda_r),
 \qquad 1\leq r\leq k.
\end{equation}
His planar two-point result is
\begin{equation}\label{eq:steiner-two-point}
 2\lambda_m+\lambda_{2m}\geq\frac{10\pi m}{V}.
\end{equation}
The first estimate retains spectral-gap information; the second couples
different indices. Neither has the same content as a volume-only lower
bound at every index.

For a quantitative comparison, put
$t_{r,k}=(r/k)(1-\lambda_r/\lambda_k)$.
Then \cref{eq:steiner-gap} is equivalently
$\lambda_k\geq a_nW_k/(1-t_{r,k})$. Its right-hand side exceeds
$c_nW_k$ precisely when
\begin{equation}\label{eq:gap-comparison}
 t_{r,k}>1-\frac{a_n}{c_n}.
\end{equation}
An independently known lower bound for $t_{r,k}$ can therefore turn
the implicit estimate into a stronger explicit one.

Conversely, applying our planar estimate at both indices only yields
\[
 2\lambda_m+\lambda_{2m}\geq
 \frac{16\pi c_2m}{V}
 =\frac{8.6129089232\ldots\,\pi m}{V},
\]
which is weaker than \cref{eq:steiner-two-point}. The uniform individual
gain therefore does not subsume the two-point theorem. Nor can one
deduce a stronger lower bound for $\lambda_k$ from
\cref{eq:steiner-gap} merely by substituting a lower bound for $\lambda_r$:
solving for $\lambda_k$ gives a negative coefficient on $\lambda_r$.

\subsection{Gan--Jiang--Lin: explicit inertia terms and a P\'olya subsequence}

For bounded open sets, Gan--Jiang--Lin~\cite[Corollary 2.7]{GJL2025}
give the following implicit estimate with explicit coefficients
\begin{equation}\label{eq:gjl-comparison}
 k\leq a_n^{-\frac{n}{2}}\alpha_n\lambda_k^{\frac{n}{2}}
 \left(1-\frac{D_n}{\lambda_k}\frac{V}{I(\Omega)}\right),
 \qquad D_n=\frac{nC_1(n)}{4(n+2)},
\end{equation}
where $C_1(n)$ is specified in their Theorem 1.1.
In particular $C_1(2)=157/480$, and the planar consequence is
\begin{equation}\label{eq:gjl-planar}
 \lambda_k\geq\frac{2\pi k}{V}
       +\frac{157}{3840}\frac{V}{I(\Omega)}.
\end{equation}
They also prove that infinitely many Dirichlet eigenvalues satisfy
P\'olya's inequality when $n\geq3$. This gives coefficient one along
a subsequence, while our coefficient $c_n<1$ applies to all indices.

Using the inertia inequality proved above, the increment in
\cref{eq:gjl-planar} is at most $157\pi/(1920V)$.
Since
\[
 4(c_2-\tfrac12)=0.1532272308\ldots
 >\frac{157}{1920}=0.0817708333\ldots,
\]
our planar bound is strictly stronger than \cref{eq:gjl-planar}
for every $k\geq1$. This comparison does not apply to their full
Riesz-mean theorem or their coefficient-one subsequence result.
For $n\geq3$, \cref{eq:gjl-comparison} can instead be evaluated as an
implicit bound and compared with $c_nW_k$ for the particular domain.

\subsection{An epsilon-loss estimate with an explicit high-frequency threshold}
\label{sec:epsilon-comparison}

Jiang--Lin~\cite{JiangLin2026} prove that, for a bounded Lipschitz
domain and each $\varepsilon\in(0,1)$, there is an explicit
$\Lambda(\varepsilon,\Omega)$ such that
\[
 k\leq (1+\varepsilon)\alpha_n\lambda_k^{\frac{n}{2}}
 \quad\hbox{whenever }\lambda_k>\Lambda(\varepsilon,\Omega).
\]
In our normalization this is
\begin{equation}\label{eq:epsilon-loss}
 \lambda_k\geq(1+\varepsilon)^{-\frac{2}{n}}W_k
 \quad\hbox{above that threshold}.
\end{equation}
Their contribution includes quantitative Weyl remainders and a
specified threshold, not merely the qualitative existence of a
large-index regime.

Algebra gives an exact coefficient comparison:
\begin{equation}\label{eq:epsilon-criterion}
 (1+\varepsilon)^{-\frac{2}{n}}>c_n
 \quad\Longleftrightarrow\quad
 \varepsilon<c_n^{-\frac{n}{2}}-1=K_n-1.
\end{equation}
One must also enforce the threshold in \cref{eq:epsilon-loss}.
Once these two conditions hold, the Jiang--Lin estimate is stronger.
In dimension two, for instance, $\varepsilon=1/2$ gives coefficient
$2/3>c_2$, but only in its stated energy range. Taking $\varepsilon$
to zero at a fixed index is not justified, since the threshold
depends on $\varepsilon$.

The two bounds can be used together without a circular threshold
test. A sufficient, explicitly checkable condition for entering
that range is
\[
 c_nW_k>\Lambda(\varepsilon,\Omega),
\]
because our theorem then guarantees
$\lambda_k>\Lambda(\varepsilon,\Omega)$. Thus the universal estimate
also supplies a way to certify when the stronger high-frequency
estimate becomes applicable. The stronger sufficient test
$L_k^{\mathrm{hyb}}>\Lambda(\varepsilon,\Omega)$ is available from
\cref{thm:hybrid-root}. This use of a certified threshold is distinct
from the integral combination in \cref{sec:hybrid}, which uses their
lower counting estimate. Whether the latter exceeds both separate
bounds is determined by \cref{eq:hybrid-strict-maximum}.

\subsection{Ji--Luo: refined averaged bounds and the poly-Laplacian}
\label{sec:ji-luo-comparison}

Ji--Luo~\cite{JiLuo2026} refine lower bounds for averaged
Dirichlet eigenvalues by retaining additional positive terms in
polynomial expansions. Their Theorem 2.1 treats the Laplacian,
and Theorem 3.1 treats higher powers with clamped Dirichlet
conditions. The optimality described there concerns the polynomial
expansions, not the sharp universal coefficient for each eigenvalue.

For a transparent quantitative comparison, their
Corollary 2.4 gives
\begin{equation}\label{eq:ji-luo-corollary}
 \frac{S_k}{k}\geq a_nW_k+\eta_{n,k}\frac{V}{I(\Omega)},
 \qquad
 \eta_{n,k}=\frac{1}{n+2}
 \min\left\{2k^{\frac{2}{n}},\frac{1}{12}
       \left(\frac{14}{5}+\frac{n}{5^n}\right)\right\}.
\end{equation}
Since $k\geq1$ and $n\geq2$, the second entry in the minimum
is smaller than the first. Write the resulting coefficient
as $\eta_n$. Passing from the average to $\lambda_k$, our
bound is at least as strong as this corollary precisely when
\begin{equation}\label{eq:ji-luo-threshold}
 k\geq
 \left(\frac{\eta_n V^{1+\frac{2}{n}}}
 {(c_n-a_n)C_nI(\Omega)}\right)^{\frac{n}{2}}.
\end{equation}
For $n=2$, $\eta_2=3/50$. The previously proved inertia inequality
$I(\Omega)\geq V^2/(2\pi)$ bounds their additive term by
$3\pi/(25V)$. Since
\[
 4(c_2-\tfrac12)=0.1532272308\ldots>\frac{3}{25},
\]
our planar bound exceeds this particular individual corollary
for every $k\geq1$.

This comparison does not rank our bound against their full
Theorem 2.1. If $B^{\rm JL}_{n,k}(\Omega)$ denotes its complete
right-hand side, then its individual consequence is
$\lambda_k\geq B^{\rm JL}_{n,k}(\Omega)$; superiority requires
evaluating $B^{\rm JL}_{n,k}(\Omega)-c_nW_k$, with all its
parameters and terms retained.

Their higher-order problem is
$(-\Delta)^\ell u=\Lambda u$ with
$u=\partial_\nu u=\cdots=\partial_\nu^{\ell-1}u=0$ on the boundary.
These eigenvalues are not, in general, the powers
$\lambda_k^\ell$ of the ordinary Dirichlet Laplacian. Accordingly,
raising our bound to the power $\ell$ does not establish a
comparison with their poly-Laplacian theorem; that problem
requires its own quadratic-form argument.

\subsection{Scope of the comparison}

The preceding calculations support three precise conclusions. First,
$c_nW_k$ strictly improves the individual bound obtained by dividing
the Li--Yau sum inequality by $k$. Second, in dimension two it also
dominates the particular individual corollaries of the inertia-based
estimates displayed in
\cref{eq:ilyin-planar,eq:gjl-planar,eq:ji-luo-corollary}.
These statements do not compare the underlying
sum or Riesz-mean inequalities in their full strength. 
Third, the hybrid root satisfies $c_nW_k\leq L_k^{\mathrm{hyb}}<W_k$ by
\cref{thm:hybrid-root,cor:hybrid-weyl-upper}. It is strictly stronger than
the maximum of the volume-only baseline and a given competing lower
bound exactly when the sign condition in
\cref{eq:hybrid-strict-maximum} holds. 

None of these conclusions makes the present estimate universally
strongest. On Euclidean balls, the Dirichlet P\'olya theorem of
Filonov--Levitin--Polterovich--Sher~\cite{FLPSBalls2023} gives coefficient
one and is therefore stronger. The thin-product results of
He--Wang~\cite{HeWang2026} likewise give coefficient one under their
geometric hypotheses. Near-ball stability estimates
\cite{BLNP2026} can provide sharper fixed-index information when the
required spectral deficit is small. The Neumann ball theorems
\cite{FLPSNeumann2026,LiNeumann2026} concern a different boundary
condition and the opposite P\'olya inequality direction, so they are not
numerically comparable with the present Dirichlet lower bound.

The high-frequency comparison is similarly conditional. Jiang--Lin's
epsilon-loss coefficient exceeds $c_n$ when
\cref{eq:epsilon-criterion} holds, but only after its explicit threshold
has been verified. Spectral-gap, two-point, and subsequence results may
also retain information unavailable to a uniform one-point estimate.
For a given domain, independently valid lower bounds may always be
combined by taking their maximum.

The hybrid mechanism is different from such a maximum. It integrates a
lower counting estimate into the exact Riesz-mean identity before solving
the scalar moment inequality. This creates the possibility of a strict
gain over both separate inputs, while the sign test prevents that
possibility from being stated as automatic dominance. Thus the appropriate
summary is not that one method replaces the others, but that the frequency
envelope provides a transparent baseline and a compatible place at which
additional geometric or spectral information can enter.

\section{Conclusions and Further Directions}\label{sec:questions}

This paper develops a frequency-capacity approach to lower bounds for
individual Dirichlet eigenvalues. Energy orthogonality supplies the
frequency-dependent constraint
$F_k(\xi)\leq V\lambda_k/((2\pi)^n|\xi|^2)$, while the standard
$L^2$ Bessel estimate supplies the constant cap. Solving the resulting
radial moment problem yields
\[
 \lambda_k\geq c_nC_nV^{-2/n}k^{2/n},
 \qquad \frac{n}{n+2}<c_n<1,
\]
for every index on every open set of finite positive measure. The result
strictly improves the direct individual consequence of Li--Yau, preserves
the correct Weyl scaling, and requires no boundary regularity. It does not
improve the sharp leading coefficient in the Li--Yau sum inequality and
does not attain the conjectured P\'olya coefficient.

The second part of the paper retains information discarded by the
elementary closure $S_k\leq k\lambda_k$. Through the identity
$S_k=k\lambda_k-R_1(\lambda_k)$, a counting-function lower bound produces
a controlled spectral deficit. Truncation at $k-1$ makes the corresponding
root equation globally well posed: its residual is continuous and strictly
decreasing, and its unique positive root is a valid eigenvalue lower bound.
Corollary~\ref{cor:hybrid-weyl-upper} further shows that this root
lies strictly below $W_k$, yielding the explicit search interval
$[L_k^{(0)},W_k]$. 
The root is monotone in the counting input, and the sign criteria in
\cref{thm:hybrid-root} determine exactly when it improves the baseline or
the maximum of independently available bounds. Jiang--Lin's quantitative
counting estimate supplies one explicit implementation on bounded
Lipschitz domains.

The comparison section also identifies the limits of these conclusions.
The universal coefficient improves specific individual corollaries, not
the full content of sharper sum, Riesz-mean, gap, or special-domain
theorems. Likewise, a positive numerical hybrid gain certifies improvement
for the stated inputs, but does not imply dominance over every applicable
geometric estimate. Keeping these distinctions explicit is essential when
the framework is used with new spectral data.

Several extensions are natural. The zero-mean refinement in
Remark~\ref{rem:zero-mean-refinement} produces a smaller, generally nonradial
Fourier capacity and therefore points to a stronger moment problem.
Additional energy levels, Fourier-gradient constraints, torsional or
heat-content information, and certified low-energy spectral data could
provide further compatible restrictions. Within the hybrid construction, 
sharper continuous counting lower bounds can increase the root
monotonically, but it remains strictly below $W_k$ as long as the moment
function $\mathcal M_k$ is unchanged. Reaching the P\'olya bound through
this approach therefore requires a stronger moment constraint or a
different spectral closure. Any progress toward P\'olya's conjecture must ultimately
control these refinements uniformly at every index; a strict improvement
of a coefficient that remains below one is not, by itself, sufficient.

\appendix
\section{Numerical Illustration of the Hybrid Comparison}\label{sec:disk-example}

This example illustrates the strict comparison criterion in
\cref{thm:hybrid-root}. The disk provides explicit geometric constants,
but its role here is computational: sharper spectral bounds are already
known for this domain. Let $\Omega=B_1$
in $\R^2$, so $V=\pi$, $\alpha_2=1/4$, and its minimal containing
rectangle is $(-1,1)^2$. Put
\[
 C=9/4,\qquad r_D=3-2\sqrt2,\qquad
 G(t)=\frac\pi2+20\pi\log_2\frac{20\sqrt t}{\pi}.
\]
Here $r_D$ is the inradius of the complement in that rectangle.
The value $C=9/4$ is a valid, deliberately nonoptimal constant for the
two-sided boundary-layer estimate on the disk for thickness less than
its width $2$: the tubular area is $4\pi s$ for $s\leq1$ and
$\pi(1+s)^2$ for $1<s<2$.

Using \cref{eq:hybrid-jl-input} with $d=\pi/(2\sqrt t)$, define
\begin{equation}\label{eq:hybrid-disk-lower}
 \nu(t)=
 \begin{cases}
 0,&0<t\leq\pi^2/4,\\
 [t/4-(C/2)\sqrt t\,G(t)]_+,&t>\pi^2/4.
 \end{cases}
\end{equation}
This piecewise definition ensures that the geometric cutoff is always
admissible where the formula is used. With the same $C$, the upper side
of \cite[Theorem 1.5]{JiangLin2026} gives the necessary inequality
$k\leq U(\lambda_k)$, where, in the energy range used below,
\begin{equation}\label{eq:hybrid-disk-upper}
 U(t)=\frac t4+\frac{\sqrt t}{4\pi}
 \left[C(8+2\pi)
 \left(\frac{\pi r_D}{2}+20\pi\log_2\frac{20\sqrt t}{\pi}\right)
 -\frac43\right].
\end{equation}
The function $U$ is strictly increasing for $t\geq\pi^2/4$; moreover,
domain monotonicity gives $\lambda_k\geq\lambda_1((-1,1)^2)>\pi^2/4$.
Thus its large positive solution $U(L_k^{\mathrm{JL}})=k$ defines the
individual lower bound being compared here. This is the general
Lipschitz estimate with the displayed constants, not an optimization
over all convex-domain improvements in that paper.

The integral of \cref{eq:hybrid-disk-lower} can be evaluated without
frequency quadrature. Write $b=C/2$, $e=10\pi/\log2$, and
$d_0=\pi/2+20\pi\log_2(20/\pi)$. An antiderivative of its untruncated
positive branch is
\[
 A(t)=\frac{t^2}{8}-bt^{3/2}
 \left[\frac23(d_0+e\log t)-\frac{4e}{9}\right].
\]
Let $t_0$ be the zero where this branch becomes positive and let $t_c$
solve $\nu(t_c)=k-1$. For $k>1$,
\[
 J_k(L)=
 \begin{cases}
 0,&L\leq t_0,\\
 A(\min\{L,t_c\})-A(t_0)+(k-1)(L-t_c)_+,&L>t_0.
 \end{cases}
\]
At $k=70\,000\,000$, ordinary double-precision evaluation gives the
following rounded root values:
\begin{center}
\begin{tabular}{@{}lr@{}}
\toprule
Estimate & Approximate value\\
\midrule
$L_k^{(0)}$ & $1.50725906\times10^8$\\
$L_k^{\mathrm{JL}}$ from \cref{eq:hybrid-disk-upper}
 & $1.51398861\times10^8$\\
$L_k^{\mathrm{hyb}}$ & $1.74556230\times10^8$\\
\bottomrule
\end{tabular}
\end{center}
At
$L_*:=\max\{L_k^{(0)},L_k^{\mathrm{JL}}\}=L_k^{\mathrm{JL}}$,
the normalized residual $\Phi_k(L_*)/(kL_*)$ is approximately $0.13447$,
and the hybrid root exceeds the larger separate value by approximately
$15.30\%$. These decimal calculations are exploratory, not
interval-arithmetic certificates. In particular, they are not used to prove
\cref{thm:hybrid-root}. On the disk the known P\'olya estimate is
$\lambda_k\geq4k=2.8\times10^8$, which is stronger than all three
displayed values. The example only illustrates the interaction of the
specified two inputs, not spectral novelty on disks or dominance over
all known bounds.

\end{document}